\documentclass[11pt,reqno]{amsart}
\usepackage{amssymb}
\usepackage{graphicx} 
\usepackage{enumerate}
\usepackage{mathtools}
\usepackage{color,soul}
\usepackage[dvipsnames]{xcolor}
\usepackage[colorlinks=true, linkcolor=orange, citecolor=blue]{hyperref}
\usepackage[nameinlink]{cleveref}

\usepackage[export]{adjustbox}
\usepackage{url}
\usepackage{cite}
\usepackage{tikz}
\usetikzlibrary{matrix}
\usepackage{caption}
\usepackage{amsmath, amsthm}
\usepackage{mathtools}
 \usepackage{relsize}
 \usetikzlibrary{cd}
 \usetikzlibrary{decorations.pathreplacing}
 \usepackage{tikz,calc}
\usepackage{color}
\usepackage[margin=1.2in]{geometry}

\usepackage{multicol}

\usepackage{wrapfig}
\usepackage{cutwin}

\usepackage{lmodern}
\usepackage{enumitem}
\usepackage{lipsum}
\usepackage{stackengine}
\usepackage{appendix}
\usepackage{scrextend}
\newtheorem{thm}{Theorem}[section]

\newtheorem{prop}[thm]{Proposition}

\newtheorem{cor}[thm]{Corollary}

\newcommand{\theoremname}{Theorem:}

\newtheorem*{conj*}{Conjecture}

  \theoremstyle{definition}
  
  \newtheorem{claim}[thm]{Claim}
  
  \newtheorem*{claim*}{Claim}

  \newtheorem*{question*}{Question}
  \newtheorem*{answer*}{Answer}
  \newtheorem*{application*}{Application}

  \theoremstyle{remark}
  
  \newtheorem*{rmk*}{Remark}

\usepackage{bbm}

\newcommand{\M}{\mathcal{M}}

\newcommand{\rpt}{\mathbb{R}\text{P}^2}
\title{Large 1-systems of Curves in non-orientable surfaces}

\begin{document}

\author{Sarah Ruth Nicholls}
\address{Wake Forest University}
\email{sarahruthnicholls@gmail.com}
%\urladdr{ADD WEBSITE IF YOU HAVE ONE}

\author{Nancy Scherich}
\address{University of Toronto}
\email{nancy.scherich@gmail.com}
\urladdr{\url{http://www.nancyscherich.com}}

\author{Julia Shneidman}
\address{Rutgers Univeristy}
\email{julshn16@gmail.com}
%\urladdr{ADD WEBSITE IF YOU HAVE ONE}

\maketitle
\begin{abstract}

     A longstanding avenue of research in orientable surface topology is to create and enumerate collections of curves in surfaces with certain intersection properties. We look for similar collections of curves in non-orientable surfaces. A surface is \emph{non-orientable} if and only if it contains a Möbius band.  We generalize a construction of Malestein-Rivin-Theran to non-orientable surfaces to exhibit a lower bound for the maximum number of curves that pairwise intersect 0 or 1 times in a generic non-orientable surface.

\end{abstract}

\section{Introduction}

The study of surfaces is often divided into two genres; orientable and non-orientable, with orientable surfaces being the far more studied and intuitive of the two.
A surface is non-orientable if and only if it contains a Möbius band. 
The Klein bottle and the real projective plane are basic examples of non-orientable surfaces, and a generic compact connected non-orientable surface is homeomorphic to a connect sum of $g$ real projective planes with $b$ boundary components, denoted $N_{g,b}$.
Many surface attributes do not depend on orientability, such as Euler Characteristic, while other properties of orientable surfaces are still true for non-orientable surfaces but require a different method of proof or a tweaked definition.
For example, the pants complex of non-orientable surfaces is connected, analogous to the orientable case, however there is a larger set of generating moves to relate two non-orientable pants decompositions \cite{pantspaper}.

A longstanding avenue of research in orientable surface topology is to create and enumerate collections of curves in surfaces with certain intersection properties.
A \textit{$k$-system} of curves in a surface is a collection of homotopy classes of simple, non-peripheral,  essential, pairwise non-homotopic curves that pairwise intersect at most $k$ times. 
For two simple closed curves $\alpha$ and $\beta$, the (geometric) \emph{intersection number} is the minimum of $|a\cap b|$ over all curves $a$ homotopic to $\alpha$  and $b$ homotopic to $\beta$. 
The work of Juvan-Malnic-Mohar \cite{JMM}, and Aougab \cite{T2012} show upper and lower bounds for the maximal size of a $k$-system in an orientable surface.
Greene \cite{G}, Aougab \cite{T}, and Malestein-Rivin-Theran \cite{MRT} more specifically found upper and lower bounds for the maximal size of a $1$-system in orientable surfaces.
It is natural to ask what the analogous bounds are for systems of curves in non-orientable surfaces.
In this paper we generalize the construction of Malestein-Rivin-Theran to find a large $1$-system in a generic compact connected non-orientable surface, $N_{g,b}$.\\

\noindent \textbf{Theorem A.} \textit{There exists a 1-system of curves, $\Gamma$, in $N_{g,b}$ with }
\begin{align*}
    |\Gamma|&\geq\frac{1}{3}g^2+\frac{5}{18}g-1 +\frac{b}{3}(g-2).
\end{align*}

\vspace{5mm}

This result gives the tightest known lower bound for the maximal size of a $1$-system in a generic compact connected non-orientable surface.
One major difference between curves in orientable versus non-orientable surfaces is that non-orientable surfaces contain \emph{1-sided curves}. A curve is \emph{1-sided} if it is the core of a Möbius band.
Altering the construction of the 1-system in Theorem A, we create a large 1-system of only 1-sided curves.\\

\noindent \textbf{Theorem B.} \textit{
There exists a 1-system, $\Omega$, of 1-sided curves in $N_{g,b}$ with }

$$|\Omega|=(g-2\lfloor \frac{g}{4}\rfloor)(2\lfloor \frac{g}{4}\rfloor+2)\geq \frac{1}{4}(g^2+3g+2).$$

\vspace{5mm}

The organization of this paper is as follows.
 Section 2 consists of background information on non-orientable surfaces and conventions for curves in non-orientable surfaces. 
In Section 3, we give intuition about curves in non-orientable surfaces and specifically address how to apply the change of coordinates principle.
Lastly, Section 4 contains the constructions of the 1-systems in Theorems A and B.\\

\noindent \textbf{Acknowledgements.} This paper is a result of a 2021 Georgia Institute of Technology REU project mentored by the second author and directed by Dan Margalit. We would like to thank Dan Margalit, Wade Bloomquist and Yvon Verberne for helpful conversations. This project was partially funded by NSF Grant DMS-1851843 and NSERC Grant RGPIN-2018-04350.

\section{Background on non-orientable surfaces}

A  surface is \emph{non-orientable} if it is impossible to make a consistent choice of normal a vector at every point on the surface (it is not possible to consistently tell the difference between clockwise and counterclockwise). 
A convenient restatement is that a surface is non-orientable if and only if it contains a Möbius band. 
Figure \ref{fig:exampleSurfaces} shows the basic examples of three non-orientable surfaces: the Möbius band, $\mathcal{M}$, real projective plane, $\rpt$, and Klein bottle, $\mathcal{K}$.

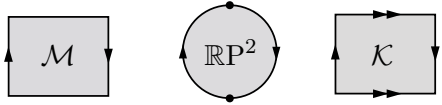
\begin{figure}[ht]
    \centering
    \begin{picture}(200,40)
    \put(10,0){\includegraphics[scale=.11]{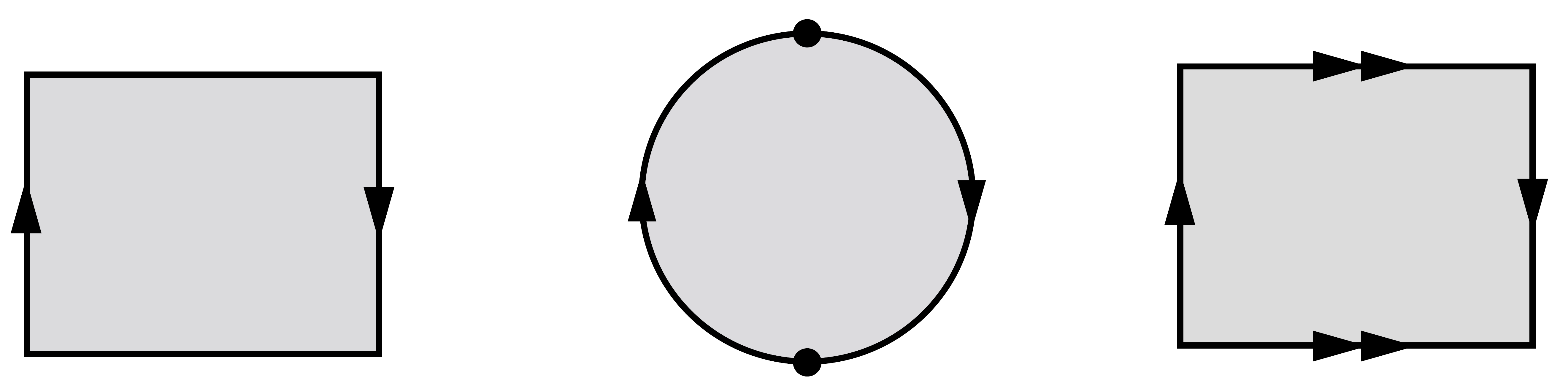}}
    \put(25,15){$\mathcal{M}$}
     \put(87,15){$\rpt$}
    \put(150,15){$\mathcal{K}$}
    \end{picture}
    \caption{The polygonal decompositions for Möbius band, real projective plane,  and Klein bottle.}
    \label{fig:exampleSurfaces}
\end{figure}

$N_{g,b}$ denotes the non-orientable surface that is the connect sum of $g$ real projective planes and has $b$ boundary components. 
From the classification of surfaces, every closed, compact, connected, non-orientable surface is homeomorphic to  $N_{g,b}$ for some $g,b\in \mathbb{Z}$ (if $b=0$ it is often omitted). 
As tori are the basic building blocks to construct orientable surfaces, $\rpt$'s are the basic building blocks to construct non-orientable surfaces.

 The following basic facts can be proved using cut-and-paste arguments:
\begin{enumerate} 
\item $N_{1,1}\cong \rpt - D^2 \cong \mathcal{M}$; the Möbius band is homeomorphic to $\rpt$ minus a disc. See Figure \ref{fig:projective plane minus disk}.
\item $N_{2,0}\cong \rpt \# \rpt \cong \mathcal{K}$; the Klein bottle is homeomorphic to $\rpt$ connect sum $\rpt$.
\item $N_{3,0}\cong \rpt \# \rpt \# \rpt \cong \mathcal{K} \# \rpt \cong S_{1,0} \# \rpt$; the Klein bottle connect sum $\rpt$ is homeomorphic to the torus connect sum $\rpt$. 

\end{enumerate}

Fact (3) says that in the presence of an $\rpt$, two non-orientable genus can be traded for one orientable genus.
It follows that all closed, compact, connected, non-orientable surfaces are homeomorphic to an orientable surface connect sum some $\rpt$'s: $$N_{g,b}\cong S_{k,b} \#^{g-2k} \rpt  \text{ for any $k<\frac{g-1}{2}$}.$$

Since $\rpt$ minus a disc is the Möbius band, connect summing with $\rpt$ can be thought of as gluing in a Möbius band along its circular boundary.
A Möbius band glued to a surface is called a \emph{cross-cap} and is drawn on the surface as $\otimes$, see Figure \ref{fig:projective plane minus disk}.

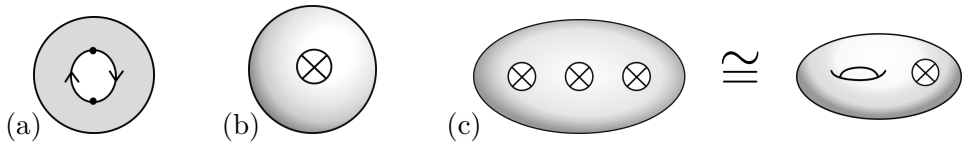
\begin{figure}
\centering
\begin{picture}(340,45)
   \put(0,0){
\includegraphics[scale=.2]{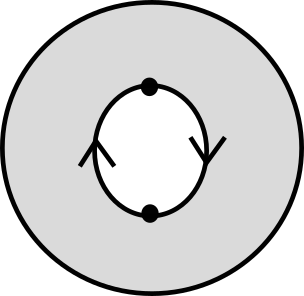}\hspace{5mm}}
\put(85,0){\includegraphics[scale=.2]{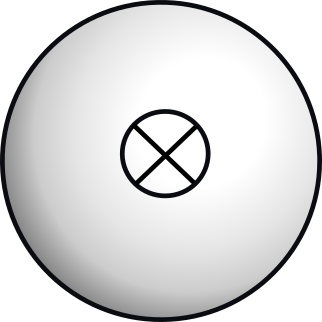}\hspace{5mm}}
\put(170,0){\includegraphics[scale=.2]{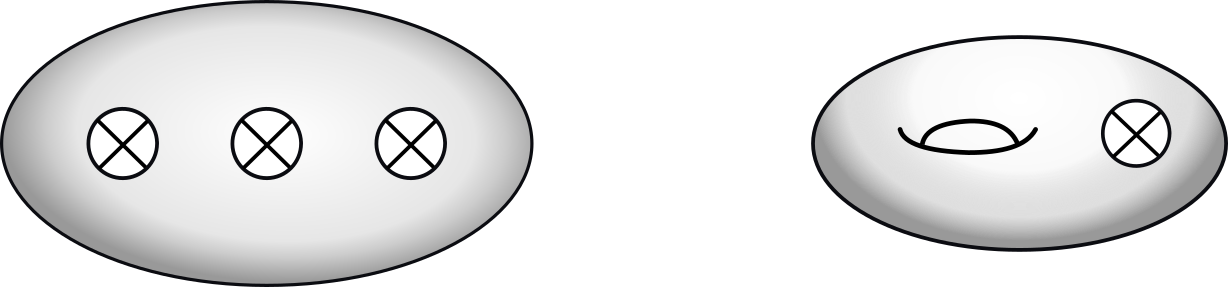}}
\put(-7,0){(a)}
\put(75,0){(b)}
\put(160,0){(c)}
\put(263,17){\huge $\cong$}
\end{picture}
\caption{(a) A depiction of $\M$ as $\rpt$ minus a disc, also called a cross-cap. (b) $\rpt$ drawn as a cross-cap on a sphere. (c) Two depictions of $N_3$.}
\label{fig:projective plane minus disk}
\end{figure}

\subsection{Conventions  for curves in non-orientable surfaces.}
All curves in this paper are closed loops in the surface and are considered up to homotopy.
A curve is \emph{simple} if it has no self-intersections, and \emph{non-peripheral} if it is not homotopic to a boundary
component.
A curve is \emph{2-sided} if it is the core of an annulus.
In orientable surfaces every curve is 2-sided.
 In non-orientable surfaces it is possible for a curve to be the core of a Möbius band and we call such curves \textit{1-sided}.
 A curve in a surface is \emph{essential} if it is not null-homotopic and not homotopic to the boundary of a Möbius band.
For example, a curve that encircles a cross-cap is considered a non-essential 2-sided curve.
For a curve to be 1-sided, it must pass through an odd number of cross-caps, see Section \ref{sec:CurvesInM}.
A curve that passes through an even number of cross-caps is 2-sided.

A \textit{$k$-system} of curves in a surface is a collection of homotopy classes of simple, non-peripheral,  essential, pairwise non-homotopic curves that pairwise intersect at most $k$ times. 
For two simple, closed curves $\alpha$ and $\beta$, the (geometric) \emph{intersection number} is the minimum of $|a\cap b|$ over all curves $a$ homotopic to $\alpha$  and $b$ homotopic to $\beta$.

 One reason for considering the boundary of a Möbius band to be non-essential is for the purposes of a pants decomposition. 
Every maximal 0-system of curves defines a pants decomposition. 
 If a curve in such a system bounds a Möbius band,  a connected component would be a Möbius band which cannot be further broken down into a pair of pants.

\section{How many curves are in non-orientable surfaces?}

In an orientable surfaces of genus greater than 0, up to homotopy, there are infinitely many different non-separating, essential, simple, closed curves, but 
up to homeomorphism, there is only one non-separating, essential, simple, closed curve.
That is, the change of coordinates principle states that there is a homeomorphism of the surfaces that maps any non-separating, essential, simple, closed curve to any other non-separating, essential, simple, closed curve. 

In non-orientable surfaces, the situation is more complicated. For example, 
 there is no homeomorphism of the surface that maps a 1-sided curve to a 2-sided curve.
However, we can ask if there is a change of coordinates so that any 1-sided curve is mapped to any other 1-sided curve. 
As it turns out, if a surface has non-orientable genus greater than or equal to 3, the answer is no!
Cutting a non-orientable surface along an essential, non-separating, simple, closed curve, can result in a non-orientable or, surprisingly, a non-orientable surface.
This phenomenon does not happen in orientable surfaces as there is no way to introduce non-orientability by cutting, but it is possible to ``cut out" non-orientability, which we show in Proposition \ref{prop:111Curve}.

From the perspective of homotopy as well as change of coordinates,  how many curves are in non-orientable surfaces?
We start with small non-orientable genus surfaces and work our way up to generic surfaces.

\subsection{Curves in the Möbius band.}\label{sec:CurvesInM}

There are two homotopy classes of simple, closed, essential curves in $\mathcal{M}$; the 1-sided core curve, and the 2-sided boundary curve. 
This fact is a corollary of the following Proposition \ref{prop:OneCurveInM}, which was proven by Epstein \cite{E} and Paris \cite{P}, and we provide a brief proof here.

\begin{prop}\label{prop:OneCurveInM}
Given any two points $x_1,x_2$ on the boundary of the Möbius band,  there is exactly one homotopy class of essential arcs connecting $x_1$ and $x_2$.
\end{prop}

\begin{proof}

Choose two points, $x_1$ and $x_2$, on the boundary circle of $\M$ and a 
 polygonal decomposition of $\M$ so that the two points are on the same side of the polygon. Up to homotopy, there are three ways for an arc to connect $x_1$ to $x_2$ by wrapping 0, 1 or 2 times around the Möbius band, as shown in Figure \ref{fig:three curves}(a)-(c). 
 As Figure \ref{fig:three curves}(d) demonstrates, an arc wrapping around the Möbius band 3 or more times must have a self-intersection. 
 The arc in Figure \ref{fig:three curves}(a) is clearly boundary parallel, and the arc in Figure \ref{fig:three curves}(c) is shown to be boundary parallel by the cut-and-paste construction in Figure \ref{fig:three curves}(e). 
 Therefore up to homotopy, the arc in Figure \ref{fig:three curves}(b) is the only essential arc.
 
 \end{proof}

 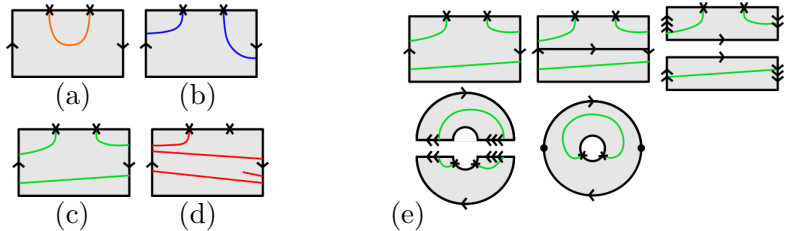
\begin{figure}[h]
\centering
  \begin{picture}(280,75)
\put(0,50){\adjincludegraphics[scale=.3,trim={0 0 {.5\width} 0},clip]{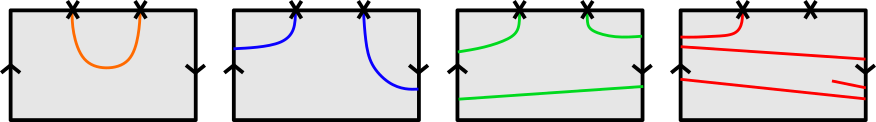}}
\put(0,5){\adjincludegraphics[scale=.3,trim={{.5\width} 0 0 0},clip]{mobiusstripcurves.png}}

\put(150,0){\includegraphics[scale=.3]{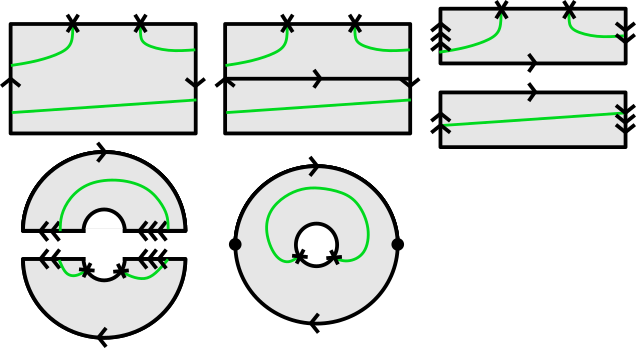}}
\put(18,40){(a)}
\put(65,40){(b)}
\put(18,-5){(c)}
\put(65,-5){(d)}
\put(145,-5){(e)}
\end{picture}
\caption{Arcs in a Möbius band.}
\label{fig:three curves}
\end{figure}

\begin{cor}
 Up to homotopy, there is only one way for a curve to pass through a cross-cap.
\end{cor}

On a surface with a cross-cap, the core curve of the Möbius band passes through the cross-cap once and is shown in Figure \ref{fig:CrossCapVisualization}(a). This curve is called \emph{the core curve of the cross-cap}.
In a neighborhood of a cross-cap, all curves look the same; they enter and exit the cross-cap at antipodal points, as shown in Figure \ref{fig:CrossCapVisualization}(b). 
 Many curves can pass through the same cross-cap and these curves remain disjoint throughout a neighborhood of the cross-cap.

\begin{cor}\label{cor:1sidedEssential}
 Any curve that passes through a cross-cap exactly once is essential.
\end{cor}

\begin{proof}
Let $\gamma$ be any curve that passes through a cross-cap once. Then $\gamma$ intersects the core curve of the cross-cap exactly once. 
By the bigon criterion, $\gamma$ must be essential.
\end{proof}

 \begin{figure}[h]
    \centering
    \begin{picture}(400,50)
    \put(0,0){\includegraphics[scale=.3]{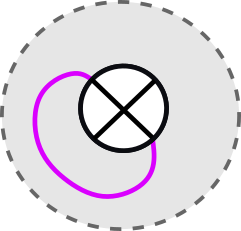}}
    \put(-7,0){(a)}
\put(80,0){\includegraphics[scale=.3]{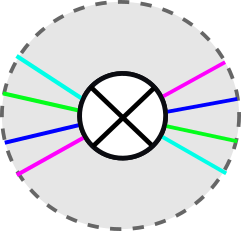}}
\put(72,0){(b)}
   \put(165,5){\adjincludegraphics[scale=.5,trim={0 0 {.70\width} 0},clip]{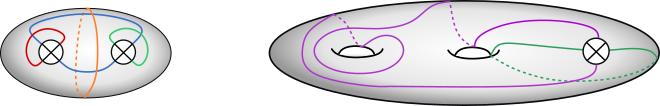}}
   \put(152,0){(c)}
     \put(245,5){\adjincludegraphics[scale=.5,trim={{.35\width} 0 0  0},clip]{CurvesinKleinandBig.png}}
     \put(240,0){(d)}
\end{picture}
    \caption{(a) The core curve of a cross-cap. (b) Multiple curves remaining disjoint as they pass through a cross-cap. (c) The 4 homotopy classes of curves in Klein bottle that do not bound discs. (d) Two disjoint curves in $N_5$.}
    \label{fig:CrossCapVisualization}
\end{figure}
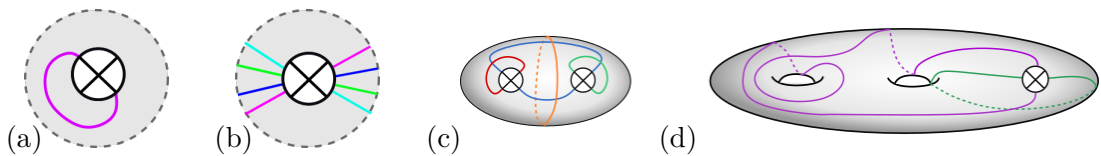

\subsection{Curves in $\rpt$ and the Klein bottle}
Both $\rpt$ and the Klein bottle have small orientable genus which limits the number of curves in the surface. In $\rpt$, the core curve of the cross-cap is the only essential, simple, closed curve up to homotopy. 
In the Klein bottle, there are exactly three simple, closed, essential homotopy classes of curves, and one non-essential curve as the boundary of the cross-cap \cite{L,P}. These curves are shown in Figure \ref{fig:CrossCapVisualization}(c).
Unlike $\rpt$, the Klein bottle has an essential 2-sided curve, however, Dehn twisting along this curve does not generate any new curves.

\subsection{Curves in $N_{g\geq 3}$.}
When $g\geq 3$, there are infinitely many homotopy classes of essential, simple, closed curves in $N_g$, see Figure \ref{fig:CrossCapVisualization}(d) for example curves in $N_5$. There exists separating and non-separating, essential, 2-sided curves, as in Figure \ref{fig:SeparatingVsNon}(a)(b). The mapping class group of a compact non-orientable surface is generated by Dehn twists and a second type of map, called a \emph{cross-cap slide}, which is depicted in Figure \ref{fig:SeparatingVsNon}(c)  \cite{P, S, L63}.

\begin{figure}[h]
    \centering
    \begin{picture}(320,80)
   \put(0,0){\includegraphics[scale=.6]{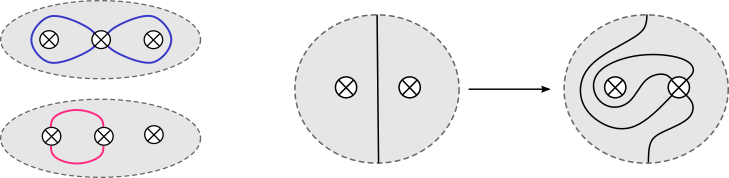}}
   \put(-9,0){(b)}
   \put(-9,45){(a)}
   \put(135,0){(c)}
   \end{picture}
    \caption{(a) A separating 2-sided curve in $N_{g\geq 3}$. (b) A non-separating 2-sided curve in $N_{g\geq 3}$. (c) Cross-cap slide}
    \label{fig:SeparatingVsNon}
\end{figure}
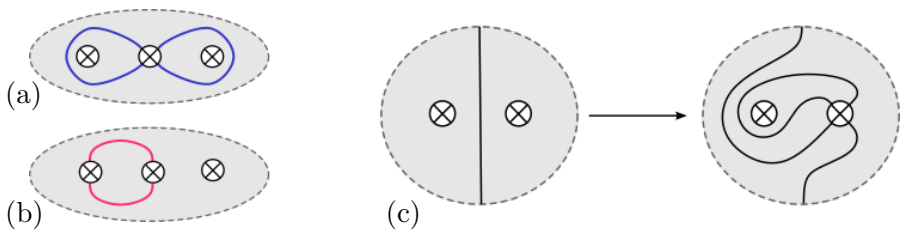

\subsection{ When are two curves related by a change of coordinates?}
A curve can be mapped to another curve if the resulting surface after cutting along one curve is homeomorphic to the result of cutting along the other curve.
By cutting along a 2-sided curve, the resulting surface has two boundary components and can be connected or disconnected.
By cutting along a 1-sided curve, the result is always connected and has only one boundary component.

Consider $N_{2,1}$, the Klein bottle with 1 boundary component. Let $\alpha$ be the 2-sided curve that passes once through each cross-cap. 
As shown in Figure \ref{fig:CuttingCurveInK}(a), cutting along $\alpha$ results in a pair of pants -- which is, unexpectedly, orientable.

\begin{figure}[ht]
\centering
\begin{picture}(320,110)
\put(0,0){\includegraphics[scale=.3]{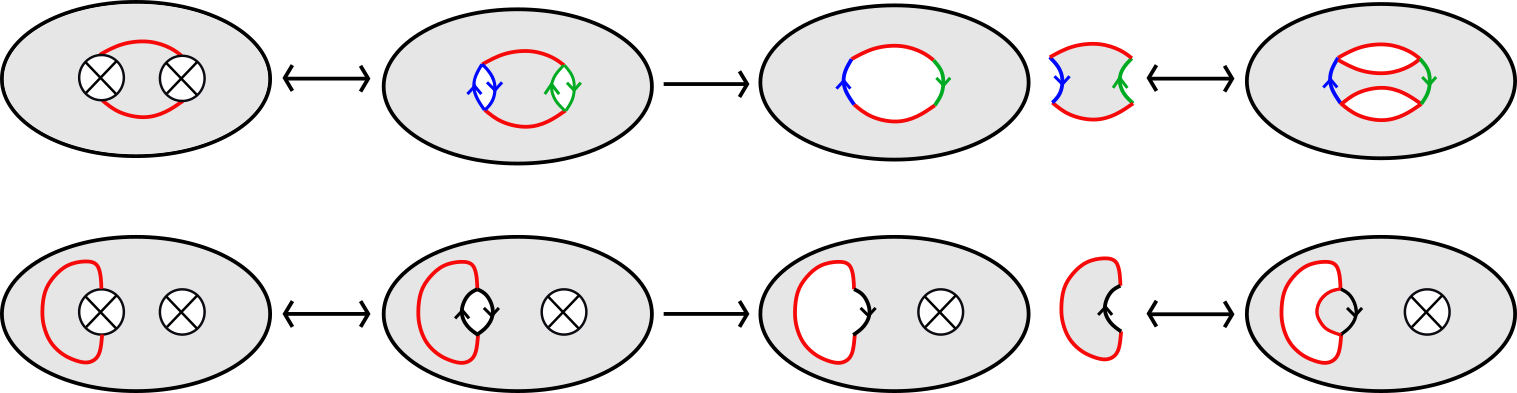}}

\put(-10,0){(b)}
\put(-10,55){(a)}
\put(30,56){$\alpha$}
\put(2,15){$\beta$}
\put(150,75){cut}
\put(150,25){cut}
\end{picture}
\caption{(a) Cutting along $\alpha$. (b) Cutting along $\beta$.}
\label{fig:CuttingCurveInK}
\end{figure}
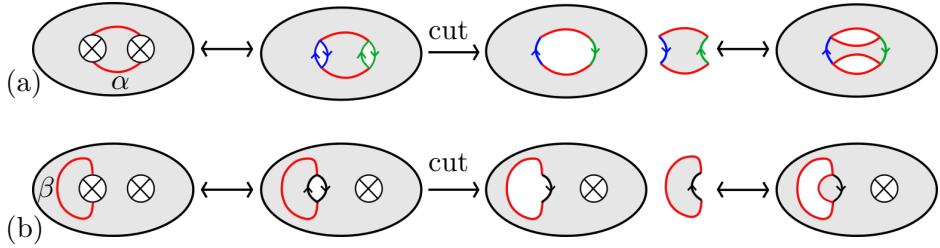

To contrast, let $\beta$ be the 1-sided curve that is the core curve of the left cross-cap, as in Figure \ref{fig:CuttingCurveInK}(b). Cutting along $\beta$ turns the left cross-cap into a boundary component--in essence this cuts out the cross-cap.
The resulting surface is homeomorphic to $\rpt$ with two boundary components and is non-orientable.
The reason that cutting along $\alpha$ yields an orientable surface, where cutting along $\beta$ yields a non-orientable surface, was \emph{not} due to the fact that the curves were 2-sided versus 1-sided. 
We could easily tell that cutting along $\beta$ yields a non-orientable surface because there is a cross-cap completely disjoint from $\beta$.
Since $\alpha$ passes through every cross-cap of the surface it is more difficult to determine the orientability of the surface after cutting along $\alpha$.

\begin{prop}\label{prop:ChangeOfCoor}
In $N_g$, there is a change of coordinates that maps any 1-sided, essential, simple, closed curve that does not intersect every cross-cap to any other 1-sided, essential, simple, closed curve that does not intersect every cross-cap. There is a change of coordinates that maps any non-separating, essential, simple, closed 2-sided curve that does not intersect every cross-cap to any other non-separating, essential, simple, closed 2-sided curve that does not intersect every cross-cap.  
\end{prop}

\begin{proof}
Let $c$ be a curve in $N_g$ that does not pass through at least one cross-cap.
Let $\Sigma$ be the resulting surface after cutting along $c$.  
$\Sigma$ will always be non-orientable since there is a cross-cap in the complement of $c$.
Since Euler Characteristic is preserved when cutting along a simple closed curve, $\Sigma$ has Euler Characteristic $2-g$. 

If $c$ is 1-sided, $\Sigma$ will have 1 boundary component and thus $\Sigma\cong N_{g-1,1}$.

If $c$ is non-separating and 2-sided, $\Sigma$ will have 2 boundary components and thus $\Sigma\cong N_{g-2,2}$.

\end{proof}

For curves that do pass through every cross-cap in the surface, it is difficult to tell the orientability after cutting along the curve. 
The following Proposition describes a type of curve that will always yield an orientable surface after cutting along it.

\begin{prop}\label{prop:111Curve}
Let $c$ be a curve in $N_{g}$ that passes through every cross-cap exactly once. The result of cutting along $c$  yields an orientable surface.
\end{prop}

\begin{proof}

It suffices to show that there are no 1-sided curves in $N_{g}$  disjoint from $c$.
 
Let $\beta$ be a 1-sided curve on the surface disjoint from $c$. 
Since $\beta$ is 1-sided, it must pass through at least one cross-cap. 
Shrinking each cross-cap to a single point, both $c$ and $\beta$ can be viewed as planar curves that intersect at some of the cross-cap points. Since $c$ is a simple closed curve in the plane, by the Jordan curve theorem, $c$ separates the plane into two components. 
The curve $\beta$ may not be simple as it might pass through the same cross-cap point more than once. 
Nevertheless, choosing any orientation of $\beta$ in the plane, the mod 2 intersection count between $\beta$ and $c$ is well defined, and, by the Jordan curve theorem, must be zero. 
Thus $\beta$ must pass through an even number of cross-cap points which implies that $\beta$ must have been a 2-sided curve to begin with. 
This shows there is no 1-sided curve in $N_{g}$ disjoint from $c$.

\end{proof}

There are curves which pass through every cross-cap in a surface, yet the result of cutting along that curve is a non-orientable surface. Figure \ref{fig:CounterExamples} shows examples of three 2-sided curves in $N_4$, each passing through all 4 cross-caps. Two of the curves are related by a change of coordinates, but the third curve is not related to the first two curves.\\

\begin{figure}
\begin{picture}(400,60)% The (400,60) is the size of the picture
   \put(0,0){\includegraphics[width=400\unitlength]{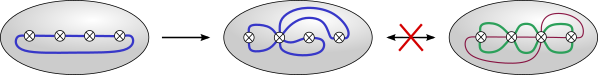}}
    \put(103,43){cross-cap}
     \put(113,30){slide}
     \put(45,7){$c_1$}
     \put(200,5){$c_2$}
     \put(328,38){$c_3$}
    \end{picture}
    \caption{Three non-homotopic 2-sided curves, $c_1$, $c_2$ and $c_3$,  in $N_4$. The curves $c_1$ and $c_2$ are related by a cross-cap slide \cite{S}. There is no homeomorphism of the surface relating the $c_3$ to either of the first two curves because there is a 1-sided curve (the thin red curve) in the complement of the $c_3$. Therefore cutting along $c_3$ yields a non-orientable surface.}
    \label{fig:CounterExamples}
    
\end{figure}
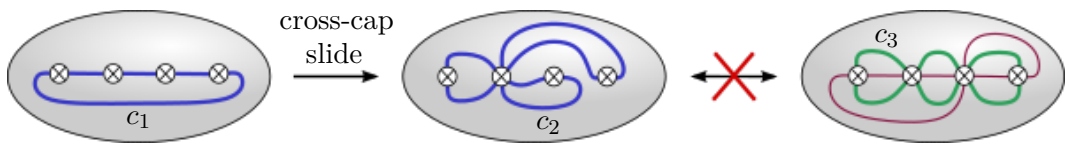

\noindent \textbf{Question:} Is there a combinatorial way to determine the orientability of the resulting surface after cutting along a curve in $N_g$?\\

Combining the results of Propositions \ref{prop:ChangeOfCoor} and \ref{prop:111Curve} with the example from Figure \ref{fig:CounterExamples}, such a curve must pass through every cross-cap, and it could pass through multiple cross-caps multiple times. 
Utilizing cross-cap slides and the mapping class group of the surface \cite{S,P,L63} perhaps a more detailed count of mod 2 cross-cap intersections could be used to determine orientability after cutting along a curve.

\section{Constructing large 1-Systems of Curves}

Recall that a 1-system of curves is a collection of simple, non-peripheral, essential, pairwise non-homotopic curves that pairwise intersect at most 1 time.
Malestein, Rivin, and Theran (MRT) give a construction for a 1-system of curves in an orientable surface of genus $g$ with $g^2+\lfloor\frac 52 g\rfloor+1$ curves  \cite{MRT}. 
This construction gives the largest known lower bound for the maximum number of curves in a 1-system in an orientable surface. 
In this section, we generalize MRT's construction to give a lower bound for the maximum number of curves in a 1-system in a non-orientable surface.
We summarize MRT's construction and show a new inductive step that applies only to non-orientable surfaces.\\

\subsection{MRT's Construction}

To construct a  1-system in an orientable surface of genus $g$ with $b$ boundary components, start with a genus $m=\lfloor \frac g2 \rfloor$ surface, $S_m$.
Choosing a polygonal decomposition of $S_m$ as a $4m$-gon, there are $2m$ simple closed curves connecting opposite sides of the polygon, and one curve connecting the diagonals of the polygon, see Figure \ref{fig:MRTFirst step}. 

\begin{figure}[ht]
    \centering
    \begin{picture}(340,70)
        \put(10,0){\includegraphics[scale=.5]{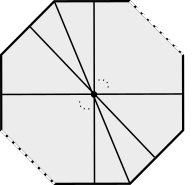}\hspace{1.5cm}}
        \put(-10,0){(a)}
        \put(-5,30){$C_1$}
        \put(150,0){\includegraphics[scale=.5]{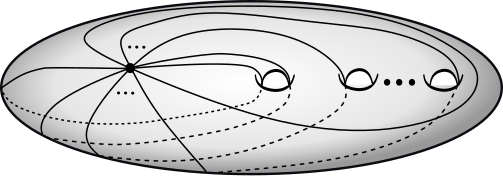}}
        \put(130,0){(b)}
    \end{picture}
    \caption{(a) $2m+1$ curves in the $4m$-gon decomposition of $S_m$. (b) A possible choice for the  $2m+1$ curves in $S_m$ as the start of MRT's construction.}
    \label{fig:MRTFirst step}
\end{figure}
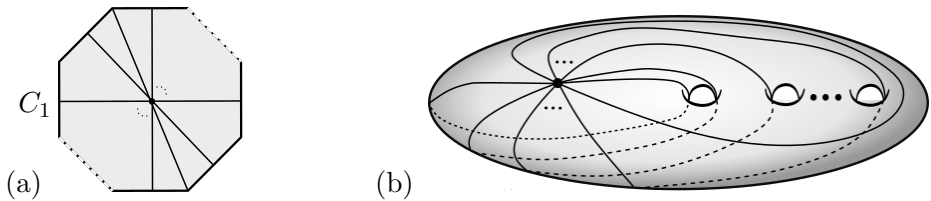

Up to homotopy, these curves can be chosen so they all intersect in the same point, $p_1$. 
This gives a 1-system of curves consisting of $2m+1$ non-homotopic, essential, non-peripheral curves in $S_m$, all of which intersect at one point.
Denote this set of curves by $C_{1}$.\\
\begin{figure}[ht]
    \centering
    \begin{picture}(350,75)
        \put(5,0){\includegraphics[scale=.6]{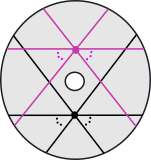}\hspace{.5cm}}
        \put(-5,0){(a)}
        \put(95,5){\includegraphics[scale=.7]{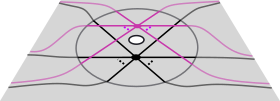}\hspace{.5cm}}
        \put(75,0){(b)}
        \put(250,20){\includegraphics[scale=.5]{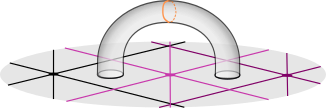}}
        \put(255,0){(c)}
        \put(-5,17){$C_1$}
        \put(35,75){$D$}
        \put(-6,50){\textcolor{magenta}{$C_2$}}
        \put(35,10){$p_1$}
        \put(35,55){$p_2$}
    \end{picture}
    \caption{MRT's Inductive Step.}
    \label{fig:DiscWithShiftedCurves}
\end{figure}
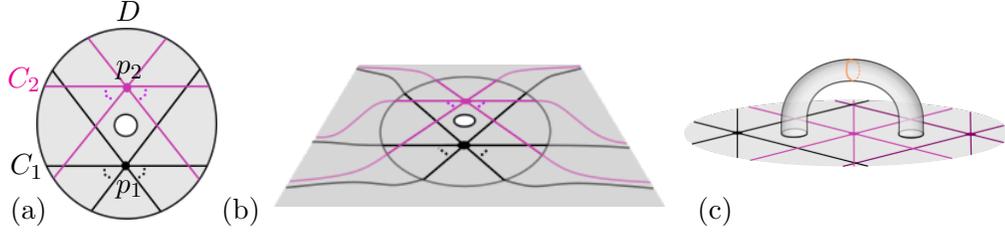

\noindent \textbf{MRT's Inductive Step:} Up to homotopy, choose a small disc neighborhood $D$ of $p_1$ where every curve in $C_1$ is a straight line arc in $D$ through $p_1$ and with no arc in the vertical direction. 
 As shown in Figure \ref{fig:DiscWithShiftedCurves}(a), within the disc, copy and shift the arcs in the vertical direction creating a new set of arcs, $C_2$, which intersect at a new point $p_2$. 
 Remove a disc between $p_1$ and $p_2$ that is disjoint from the arcs to create a new boundary component.
 Inside $D$, every arc in $C_1$ intersects every arc in $C_2$ except the one parallel vertical translate of the arc itself.
 
  Outside of $D$, each arc in $C_2$ can be completed into a simple closed curve by following closely the one parallel vertical translate curve in $C_1$, as shown in Figure \ref{fig:DiscWithShiftedCurves}(b).
  Thus, $C_2$ is a set of curves in $S_{m,1}$ not just a set of arcs in $D$.
  By construction, all of the curves in $C_1 \cup C_2$ are essential, pairwise non-homotopic, non-peripheral and form a 1-system in $S_{m,1}$ with all intersections occurring in the neighborhood $D$. 

Iterating this copy and shift process $2n+b=2(g-m)+b$ times yields a 1-family of curves $C=\bigcup_{i=1}^{2n+b} C_i$ in $S_{m,2n+b}$ with $(2m+1)(2n+b+1)$ curves.
For $n$ pairs of boundary components, glue in a cylinder handle to connect the boundary components, as in Figure \ref{fig:DiscWithShiftedCurves}(c). 
Add to $C$ the meridian curve of each handle.
 This completes MRT's construction and yields a 1-system of curves, $C$, in the closed orientable surface of genus $m+n=g$ and with $b$ boundary components with 
 \begin{equation}\label{eq:SizeOfC}
     |C|=(2m+1)(2n+b+1)+n  \text{ where }n=g-m \text{ and }m=\lfloor \frac{g}{2}\rfloor .
 \end{equation}

\subsection{New Inductive Step}

We offer a different inductive ``copy and shift" step that adds a cross-cap instead of a boundary component.
Start with $C_1$, the set of $2m+1$ curves in $S_m$ that all intersect at 1 point $p_1$.
Choose a small disc neighborhood $D'$ of $p_1$ where every curve in $C_1$ is a straight line through $p_1$ and with no arc in the vertical direction. 
Copy and shift the curves from $C_1$ twice in the vertical direction. Call these new curves $\gamma_{1}$ and $\tilde{C}_{1}$, which have intersection points $g_1$ and $\tilde{p}_1$ as in Figure \ref{fig:NewDisc}.
Cut out a small disc about $g_1$ and glue in a cross-cap. 
This makes every curve in $\gamma_{1}$ a 1-sided curve.
 
Add the core curve of the cross-cap to $\gamma_{1}$.
As a result, in this new disc we have added $2m+1$ 2-sided curves denoted $\tilde{C}_{1}$ and $2m+2$ 1-sided curves in $\gamma_1$.

Just as in MRT's construction, outside of the disc, every curve follows closely to its parallel shifted version. 
Every curve intersects every other curve in the disc except for the 2 parallel shifted versions.
Outside of this disc, all curves have the same intersection properties as the set $C$ from the MRT construction.

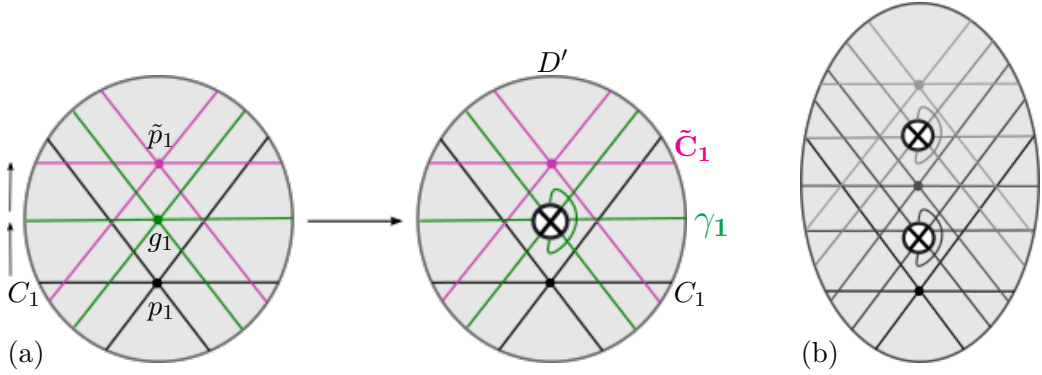
\begin{figure}
    \centering
    \begin{picture}(405,150)% The (400,60) is the size of the picture
   \put(0,0){\includegraphics[scale=.9]{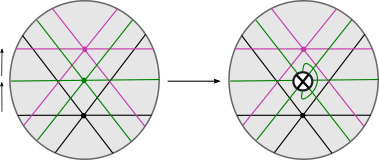}}
   \put(0,0){(a)}
   %\put(150,0){(b)}
   \put(300,0){(b)}
    \put(300,0){\includegraphics[scale=.8]{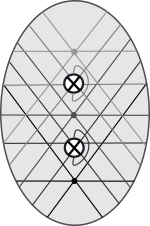}}
    \put(0,23){$C_1$}
    \put(53,18){$p_1$}
     \put(53,44){$g_1$}
     \put(53,83){$\tilde{p}_1$}
     
     \put(252,23){$C_1$}
     \put(255,50){\textcolor{Green}{\Large{ $\mathbf{\gamma_1}$}}}
     \put(252,78){\textcolor{magenta}{$\mathbf{\tilde{C}_1}$}}
     
       \put(200,110){$D'$}
    \end{picture}
    \caption{(a) Constructing curves in $D'$. (b) Iterating the new inductive step two times.}
    \label{fig:NewDisc}
\end{figure}

\begin{claim}
Every curve in $C_1\cup\gamma_1\cup \tilde{C}_1$ is essential in $S_m\# \rpt$. 
\end{claim}

\begin{proof}
Since every curve in $\gamma_{1}$ passes through a cross-cap exactly once, then every curve in $\gamma_{1}$ is essential by Corollary \ref{cor:1sidedEssential}.

By MRT's construction, every 2-sided curve in $\Gamma$  is essential in $S_m$. 
The only way for one of these curves to become inessential in $S_m\# \rpt$ is if the curve bounds a Möbius band (encircles the cross-cap).
However, through the steps of MRT's construction, every curve was non-peripheral to the newly introduced (and later filled) boundary components.
In our new construction, all of these curves are non-peripheral to the disc $D'$, which implies that none of these curves encircle the cross-cap.

\end{proof}

\begin{claim}
All curves in $C_1\cup\gamma_1\cup \tilde{C}_1$ are pairwise non-homotopic.
\end{claim} 

\begin{proof}

Let $\alpha$ and $\beta$ be two distinct curves in $C_1\cup\gamma_1\cup \tilde{C}_1$ that are not vertical shifts of each other. Let $\delta$ be a vertical shift of $\alpha$. By construction $\beta$ intersects $\delta$ exactly once, while $\alpha$ is disjoint from $\delta$. Intersection number is preserved by homotopy, so $\alpha$ and $\beta$ must be non-homotopic.

Let $\alpha$ and $\beta$ be two distinct curves in $C_1\cup\gamma_1\cup \tilde{C}_1$ that are vertical shifts of each other. If one curve is 1-sided and the other is 2-sided then the curves cannot be homotopic. 
Now, suppose both curves are 2-sided.
If the cross-cap was removed and patched with a disc, then by construction $\alpha$ and $\beta$ would be homotopic. However, the placement of the cross-cap is in the homotopy annulus bounded by $\alpha$ and $\beta$, which ultimately renders the two curves non-homotopic.

\end{proof}

\subsection {Combining MRT and the New Inductive Step.}

To construct a large 1-system of curves in $N_{g,b}$, we first do MRT's construction followed by our new inductive step.
 MRT's construction starts with a choice of $k\in \mathbb{N}$ and converts $2k$ of the $g$ cross-caps into $k$ orientable genus.
MRT's construction gives a 1-system of curves in $S_{k,b}$ called $C$ with $|C|=(2\lfloor \frac{k}{2}\rfloor+1)(2(k-\lfloor \frac{k}{2}\rfloor)+b+1)+(k-\lfloor \frac{k}{2}\rfloor)$ as in Equation \ref{eq:SizeOfC}.
From here, iterate our new inductive step $g-2k$ times which creates new sets of curves $\gamma_1,\tilde{C}_1,\cdots, \gamma_{g-2k},\tilde{C}_{g-2k}$.
Let $\Gamma=C\bigcup_{i=1}^{g-2k} \gamma_{i}\bigcup_{i=1}^{g-2k} \tilde{C}_{i}$ be the full collection of curves from the construction. 
By construction, $\Gamma$ is a 1-system of curves in $S_{k,b}\#^{g-2k}\rpt \cong N_{g,b}$ with 
\begin{equation}\label{eq:sizeOFGamma}
    |\Gamma|=|C|+(g-2k)(4\lfloor \frac{k}{2}\rfloor+3).
\end{equation}

\subsection{Choosing $k$ to maximize $|\Gamma|$}

Letting $b=0$, from Equation \ref{eq:SizeOfC}, we see that
\[
    |C|=(2\lfloor \frac{k}{2}\rfloor+1)(2(k-\lfloor \frac{k}{2}\rfloor)+1)+(k-\lfloor \frac{k}{2}\rfloor)= k^2+\lfloor \frac{5k}{2}\rfloor +1. %\text{ This is true for all $k$.}
\]

Then, from Equation \ref{eq:sizeOFGamma}, we get that
\[
    |\Gamma|=|C|+(g-2k)(4\lfloor \frac{k}{2}\rfloor+3)= k^2+\lfloor \frac{5k}{2}\rfloor +1 +(g-2k)(4\lfloor \frac{k}{2}\rfloor+3).\\
\]

This quantity is maximized when $k=\lfloor \frac{g}{3}\rfloor$, which gives the largest size of $\Gamma$ with 
\begin{align*}
    |\Gamma|&=\lfloor \frac{g}{3}\rfloor^2+\lfloor \frac{5\lfloor \frac{g}{3}\rfloor}{2}\rfloor +1 +(g-2\lfloor \frac{g}{3}\rfloor)(4\lfloor \frac{g}{6}\rfloor+3)+(2\lfloor \frac{g}{6}\rfloor+1)b\\
    &\geq \frac{1}{3}g^2+\frac{5}{18}g-1 +\frac{b}{3}(g-2).
\end{align*}

\vspace{5mm}

\noindent \textbf{Theorem A.} \textit{There exists a 1-system of curves $\Gamma$ in $N_{g,b}$ with }

\begin{align*}
    |\Gamma|&\geq\frac{1}{3}g^2+\frac{5}{18}g-1 +\frac{b}{3}(g-2).
\end{align*}
\qed

\subsection{Constructing a large 1-system of 1-sided curves.}

To construct a large 1-system of 1-sided curves in $N_{g}$, choose a $k\in \mathbb{N}$ and convert $2k$ of the $g$ cross-caps into $k$ orientable genus.
There is a collection of curves $C_1$ in $S_k$ that intersect at exactly one point just as in the start of MRT's construction.
Do the new inductive step $g-2k$ times, but only add in the 1-sided curves in the sets $\gamma_i$. 
Each set $\gamma_i$ contains $2k+2$ curves.
Then $\Omega=\bigcup_{i=1}^{g-2k}\gamma_i$ is a 1-system of only 1-sided curves in $S_{k}\#^{g-k}\rpt$ with 
\[|\Omega|=(g-2k)(2k+2).\]

This is maximized when $k=\lfloor \frac{g}{4}\rfloor$.\\

\noindent \textbf{Theorem B.} \textit{
There exists a 1-system $\Omega$ of 1-sided curves in $N_g$ with }

$$|\Omega|=(g-2\lfloor \frac{g}{4}\rfloor)(2\lfloor \frac{g}{4}\rfloor+2)\geq \frac{1}{4}(g^2+3g+2).$$\qed

\bibliographystyle{plain}
\bibliography{NonOrientableCurves}
\end{document}